\documentclass[12pt,reqno]{amsart}
\usepackage{graphicx,indentfirst}
\usepackage{amsmath,amssymb,mathrsfs}
\usepackage{amsthm,amscd}
\usepackage{verbatim}
\usepackage{appendix}
\usepackage{enumitem,titletoc}
\usepackage[utf8]{inputenc}
\usepackage{imakeidx}
\makeindex[columns=2, title=Alphabetical Index]
\usepackage{fancyhdr}
\usepackage{amsfonts,color}
\usepackage[all]{xy}
\usepackage{tikz-cd}
\usepackage{syntonly}
\usepackage{fancyhdr}
\usepackage{array}
\usepackage[left=2.5cm,right=2.5cm,bottom=3cm,top=3cm]{geometry}

\usepackage{tikz}
\usepackage{extarrows}
\usepackage{hyperref}

\def\XXint#1#2#3{{\setbox0=\hbox{$#1{#2#3}{\int}$ }
		\vcenter{\hbox{$#2#3$ }}\kern-.6\wd0}}

\newtheorem{claim}{Claim}
\newtheorem{lem}{Lemma}
\newtheorem{prop}{Proposition}[section]

\newtheorem{corr}{Corollary}[section]
\newtheorem{conj}{Conjecture}
\newtheorem{remark}{Remark}
\newtheorem{theorem}{Theorem}[section]

\allowdisplaybreaks
\newcommand{\del}{\partial}

\newcommand{\eps}{\varepsilon}
\newcommand{\ddbar}{\sqrt{-1}\partial\bar\partial}

\makeindex

\hypersetup{
	colorlinks=true,
	citecolor=green,
	filecolor=green,
	linkcolor=blue,
	urlcolor=black
}

\numberwithin{equation}{section}
\title[A free boundary Monge-Amp\`ere equation]{A free boundary Monge-Amp\`ere equation and applications to complete Calabi-Yau metrics}
\author{Tristan C. Collins}
\email{tristanc@math.toronto.edu}
\email{tristanc@math.mit.edu}
\address{Department of Mathematics, University of Toronto, 40 St. George Street, Toronto, ON, Canada}

\address{Department of Mathematics, Massachusetts Institute of Technology, 77 Massachusetts Ave., Cambridge, MA, USA}
\author{Freid Tong}
\email{ftong@cmsa.fas.harvard.edu}
\address{Center of Mathematical Sciences and Applications, Harvard University, 20 Garden St., Cambridge, MA, USA}
\author{Shing-Tung Yau}
  \email{styau@tsinghua.edu.cn}
\address{Department of Mathematics, Tsinghua University, Haidian District, Beijing 100084, China}

\begin{document}
\begin{abstract}
Let $P$ be a convex body containing the origin in its interior. We study a real Monge-Amp\`ere equation with singularities along $\del P$ which is Legendre dual to a certain free boundary Monge-Amp\`ere equation.  This is motivated by the existence problem for complete Calabi-Yau metrics on log Calabi-Yau pairs $(X, D)$ with $D$ an ample, simple normal crossings divisor.  We prove the existence of solutions in $C^{\infty}(P)\cap C^{1,\alpha}(\overline{P})$, and establish the strict convexity of the free boundary.  When $P$ is a polytope, we obtain an asymptotic expansion for the solution near the interior of the codimension $1$ faces of $\del P$.
\end{abstract}
	\maketitle
	
	\section{Introduction}

Suppose $X$ is a Fano manifold (i.e. $-K_{X}>0$), and suppose that $D\in|-K_{X}|$ is a divisor.  The complement $Y=X\setminus D$ admits a holomorphic volume form $\Omega$ which has a simple pole along $D$ and is naturally an open Calabi-Yau manifold of infinite volume.  It is a natural question, dating back to the third author's ICM address \cite{YauICM} to construct complete Calabi-Yau metrics on $Y$.  Such metrics, particularly those with sub-Euclidean volume growth, are fundamental for understanding the collapsing behavior of compact K\"ahler-Einstein manifolds; see, e.g. \cite{Biquard-Guenancia, HSVZ, SunZhang, Fu-Hein-Jiang} for some recent examples. 

The first major result concerning the existence of complete Calabi-Yau metrics is due to Tian and Yau \cite{Tian-Yau, Tian-Yau2}, who addressed the case when $D$ is smooth and irreducible.  The work of Tian-Yau develops a robust package for solving the complex Monge-Amp\`ere equation with ``good" asymptotics, extending to the non-compact case the techniques pioneered by the third author in the solution of the Calabi conjecture \cite{Yau}.  These results can be used to produce complete Calabi-Yau metrics in any case where a good ansatz metric at infinity is known and many examples are found in \cite{Tian-Yau, Tian-Yau2}.

Understanding the existence of complete Calabi-Yau metrics when $D$ is singular (eg. with simple normal crossings) is a long-standing problem in the field.  Hein \cite{Hein} constructed complete Calabi-Yau metrics on the complement of singular elliptic fibers in rational elliptic surfaces (for which $-K_X$ is effective, but not ample) utilizing the semi-flat metrics discovered by Greene-Shapere-Vafa-Yau \cite{GSVY}. Recently, the first author and Li \cite{Collins-Li} constructed complete Calabi-Yau metrics on $X\setminus D$ when $X$ is Fano, $\dim_{\mathbb{C}} X>2$, and $D=D_1+D_2 \in |-K_{X}|$ is a simple normal crossings divisor with $D_1 \equiv_{\mathbb{Q}} D_2$.  In \cite{Collins-Li} a general inductive strategy was proposed for constructing complete Calabi-Yau metrics on general Fano manifolds $X$ equipped with an snc anti-canonical divisor $D=D_1+\cdots D_k$ for $k < \dim_{\mathbb{C}}X$ and $D_i$ ample $1 \leq i \leq k$.  Roughly speaking, the basic idea of this inductive approach is that the Calabi-Yau metric on $X\setminus (D_1+\cdots +D_k)$ should asymptote near $D_i$ to the complete Calabi-Yau metric on $D_i\setminus (\sum_{j\ne i}^k D_j)$.  Technically, the inductive approach amounts to solving a certain real Monge-Amp\`ere equation with singularities along the boundary of a polytope and obtaining a complete asymptotic expansion for the solution near the boundary.  Crucially, the leading sublinear term in this expansion must be related to the solution of a related Monge-Amp\`ere equation in fewer variables.  The purpose of this paper is to address the existence of solutions to this real Monge-Amp\`ere equation and to show that near the generic point of the boundary, these solutions have leading order asymptotics determined by the asymptotics of the complete Calabi-Yau metrics constructed by Tian and Yau \cite{Tian-Yau}.    
	
To introduce our problem, let $P\subset \mathbb R^n$ be an open convex set containing the origin and consider the following boundary value problem
	\begin{equation}\label{eqn: main-equation}
		\begin{cases}
			\det D^2v = v^{-(n+2)}(-v^{\star})^{-k} & \text{ in } P\\
			v^{\star} = 0 & \text{ on } \partial P. 
		\end{cases}
	\end{equation}
	Here $v^{\star}$ is the function defined by\[v^{\star}(y):= \langle y, \nabla v(y)\rangle - v(y),\]which can be interpreted as the composition of the Legendre transform of $v$ with the gradient map $\nabla v$. Similar types of equations have been studied by the second and third authors in \cite{Tong-Yau}, where the focus was on the Dirichlet problem. The boundary condition in \eqref{eqn: main-equation} is the natural generalization of the boundary condition considered in \cite{Collins-Li}, and it asks for $v^{\star}$ to vanish on the boundary. This new boundary condition is equivalent to a free-boundary problem when re-casted in terms of the Legendre transform. Indeed if we let $u:\mathbb R^n\to \mathbb R$ be the Legendre transform of $v$,
    \[u(x) :=\sup_{y\in P}(\langle x, y\rangle-v(y)),\]
    then the equation~\eqref{eqn: main-equation} is equivalent to the following equation for $u$,
	\begin{equation}\label{eqn: legendre-transform-equation}
		\begin{cases}
			\det D^2u = (u^{\star})^{n+2}(-u)_+^{k} & \text{ in } \mathbb R^n\\
			\nabla u(\mathbb R^n) = \overline P,
		\end{cases}
	\end{equation}
	where $\chi_{\{u<0\}}$ is the characteric function of the set $\{u<0\}$. If we denote $\Omega = \{u<0\}$, then the equation~\eqref{eqn: legendre-transform-equation} is elliptic inside $\Omega$, while it is degenerate outside $\Omega$. The set $\partial\Omega = \{u = 0\}$ behaves like a free-boundary, and knowledge of $u$ inside $\Omega$ determines $u$ globally by taking the minimal convex extension of $u|_{\Omega}$. 
    \begin{remark}
    There is a separate transformation that we can apply to equation~\eqref{eqn: legendre-transform-equation}. Namely, if we define new coordinates $z\in \mathbb R^n$ by
    \[z = \frac{x}{-u(x)} \text{ for }x\in \Omega\]
    and a function of the new coordinate $\phi(z)$ by
    \[\phi(z) = \frac{1}{-u(x)},\]
    then by \cite[Section 5]{Tong-Yau}, the function $\phi:\mathbb R^n\to \mathbb R$ satisfies the Monge-Amp\`ere equation
    \[\det D^2\phi = \frac{c}{\phi^{n+2+k}}.\]
    This equation was studied by Klartag in \cite{Klartag}, where it was called the $k$-moment measure. Klartag solved this equation together with the second boundary condition. In our case, we would like $u$ to satisfy a second boundary condition, which in terms of $\phi$ is equivalent to the boundary condition \[\left\{\frac{\nabla\phi(z)}{-\phi^{\star}(z)}: z\in \mathbb R^n\right\} = P.\] 
    \end{remark}
	
	Our first main theorem is the following existence result for the Monge-Amp\`ere equation~\eqref{eqn: main-equation}. 
	\begin{theorem}\label{thm: main-theorem}
		Let $P$ be an open convex set that contains the origin, then the equation \eqref{eqn: main-equation} admits a solution $v\in C^{\infty}(P)\cap C^{1, \alpha}(\overline P)$ for some $\alpha>0$. 
	\end{theorem}
    \begin{remark}
	Notice that the right-hand side of the equation~\eqref{eqn: main-equation} is singular at the boundary, hence the solution cannot be $C^2$ up to the boundary, and $C^{1, \alpha}$ is the best one can hope for. 
    \end{remark}
        A corollary of this result is the solution of the following singular Monge-Amp\`ere equation, which can be thought of as the solutions of an optimal transport problem between non-compact domains. 
    \begin{corr}\label{cor: intro-solution-to-generalized-ansatz}
        The boundary-value problem 
        \begin{equation}\label{eq: genearlized-calabi-ansatz}
		\begin{cases}
			\left(\sum_{i=1}^{n+1}\frac{\partial \varphi}{\partial x_i}\right)^k\det D^2\varphi=1 \text{ in }(\mathbb R_+)^{n+1}\\
			\sum_{i=1}^{n+1}\frac{\partial \varphi}{\partial x_i}= 0 \text{ on }\partial (\mathbb R_+)^{n+1}. 
		\end{cases}
	\end{equation}
        admits a positive homogenous solution $\varphi\in C^{1, \alpha}(\overline{(\mathbb R_+)^{n+1}})\cap C^{\infty}((\mathbb R_+)^{n+1})$. 
    \end{corr}
 
    The situation when $P$ is a simplex is particularly important as it is related to the study of complete Calabi-Yau manifolds \cite{Collins-Li} and so we pay particular attention to the regularity of the free boundary in this case. 
    The boundary of a simplex $\partial P$ has a stratification $\partial P_0\subset \partial P_1 \subset \cdots \subset \partial P_{n-1}$ where $\partial P_k$ is the $k$ skeleton consisting of the faces up to dimension $k$.  In this case, we have the following theorem which describes the leading sublinear behavior of $v$ at the top stratum of the boundary $\partial P_{n-1}\setminus \partial P_{n-2}$. 
	\begin{theorem}\label{thm: boundary-regularity}
		In the case when $P$ is a simplex, for any point $\bold x\in \partial P_{n-1}\setminus \partial P_{n-2}$, if we let $\vec n$ be the inward-pointing normal vector to $\partial P$ at $x$ and define the local coordinate system $\tilde x_n = \vec n\cdot (x-\bold x)$, and $(\tilde x_1, \ldots, \tilde x_{n-1})$ be orthogonal coordinates that's tangential to $\partial P$, then in this coordinate system, one has the following expansion for $x$ sufficiently close to $\bold x$
        \begin{equation}
            v(x) = v(\bold x)+\nabla v(\bold x)\cdot (x-\bold x)+p\tilde x_n^{1+\frac{1}{k+1}}+\sum_{i, j = 1}^{n-1}P_{ij}(\tilde x_i-\bold{\tilde x_i})(\tilde x_j-\bold{\tilde x_j})+O\left(\left(|x-\bold x|^2+\tilde x_n^{1+\frac{1}{k+1}}\right)^{1+\frac{\alpha}{2}}\right)
        \end{equation}
        for some $\alpha>0$, where $P_{ij} = D^2v(\frac{\partial }{\partial \tilde x_i}, \frac{\partial}{\partial \tilde x_j})$ is a positive definite matrix and $p>0$ is a positive constant. 
	\end{theorem}

 The proof of Theorem~\ref{thm: main-theorem} is via a min-max argument, while Theorem~\ref{thm: boundary-regularity} is related to results of Jhaveri-Savin \cite{Jhaveri-Savin} on optimal transport with degenerate densities.  This paper is organized as follows: in Section~\ref{sec: geometricApplications} we discuss the geometric origins of equation~\ref{eqn: main-equation} and the connection with complete Calabi-Yau metrics.  We also pose a conjecture concerning a Liouville theorem for solutions to a certain optimal transport problem between non-compact domains. In Section~\ref{sec: MAFunctionals} we construct the functional whose critical points are formal solutions of~\eqref{eqn: main-equation}.  In Section~\ref{sec: proofOfMain} we prove Theorem~\ref{thm: main-theorem}.  In Section~\ref{sec: regularity} we discuss regularity issues and prove Theorem~\ref{thm: boundary-regularity}.  Furthermore, we establish the strict convexity of the free boundary in general. 

 \section{Geometric motivation}\label{sec: geometricApplications}
    In this section, we describe the relationship between the solution of the free-boundary problem \eqref{eqn: main-equation} and complete Calabi-Yau metrics. Let $X$ be a compact K\"ahler manifold of complex dimension $n+1+k$ and $D\in |-K_X|$ an ample anticanonical divisor. Assume that $D = D_1 +\ldots, D_{n+1}$ is simple normal crossing with $D_i$ ample. Then an important question is to ask whether $M = X\setminus D$ admits a complete Calabi-Yau metric.

    In \cite{Collins-Li}, it is realized that to construct such a model metric, one has to first construct the metric in the "generic region", which corresponds to a neighborhood of the top intersections $D_1\cap\cdots\cap D_{n+1}$.  The intersection $Z= D_1\cap \cdots \cap D_{n+1}$ is a compact, connected Calabi-Yau manifold by the adjunction formula and the Lefschetz hyperplane theorem.  Let us assume for simplicity that $D_i \in |d_iL_0|$ for some ample line bundle $L_0 \rightarrow X$ (see \cite{Collins-Li} for discussion on the general case).  A neighborhood of $Z \subset X$ is modeled on a neighborhood of the zero section in the bundle
    \[
    \bigoplus_{i=1}^{n+1} L_0^{d_i} \rightarrow Z.
    \]
    While such an identification is not holomorphic, the errors introduced by the identification are essentially negligible for the construction and we shall ignore them henceforth (see \cite{Collins-Li} for precise, quantitative discussion).
    By the third author's result \cite{Yau} we can find a metric $h$ on $L_0$, unique up to scaling, such that $-\ddbar \log h = \omega_{CY}$, the unique Calabi-Yau metric on $Z$ in the class $c_1(L_0)|_{Z}$.  Consider the real coordinates $x_i= -\frac{1}{d_i}\log|\sigma_i|_{h}$ where $\sigma_i$ is a local trivialization of $L_0^{d_i}$.  We now look for a function $\phi$ defined on $\{ x_i \gg 1 \forall i\}$ such that $\ddbar \phi$ defines a complete Calabi-Yau metric.  A straightforward calculation \cite{Collins-Li} shows $u$ must satisfy
    \[
    \quad D^{2}\phi >0, \quad \sum_{i=1}^{n+1} \frac{\del \phi}{\del x_i}>0
    \]
    and solve the real Monge-Amp\`ere equation
    \begin{equation}\label{eq: dimRedRMA}
    \det(D^2\phi)\left(\sum_{i=1}^{n+1} \frac{\del \phi}{\del x_i}\right)^k =c
    \end{equation}
    for some constant $c>0$, whose value may always be taken to be $1$ by rescaling $u$.  Note that for $n=0$, this equation reduces to the ODE whose solution is given by $\phi(x)= x^{\frac{k+2}{k+1}}$, which is precisely the asymptotics of the Tian-Yau metric \cite{Tian-Yau}.  It is natural to look for solutions of~\eqref{eq: dimRedRMA} which are homogeneous.  An easy calculation shows that the natural scaling yields
    \[
    \phi(x_1,\ldots,x_{n+1}) = \left(\sum_{i=1}^{n+1} x_i\right)^{1+\frac{n+1}{n+k+1}} v\left(\frac{x_1}{\sum_{i=1}^{n+1} x_i}, \ldots, \frac{x_{n+1}}{\sum_{i=1}^{n+1} x_i}\right)
    \]
    If we define variables $y_i = \frac{(n+1)x_i}{\sum_{i=1}^{n+1}x_i}-1$ for $1 \leq i \leq n$, then $v(y_1,\ldots, y_n)$ is defined on $P = \cap_{i}\{ -1 \leq y_i \} \cap \{\sum_{i=1}^n y_i \leq 1\}$, the standard simplex with barycenter at the origin. The calculation in Proposition~\ref{prop: phi-solves-gen-Calabi-ansatz} shows that~\eqref{eq: dimRedRMA} is equivalent to
    \begin{equation}\label{eq: dimRedHomMA}
    \det(D^2v) = \frac{c}{(v)^{n+2}(-v^{\star})^{k}}.
    \end{equation}
    which is precisely~\eqref{eqn: main-equation}.  The remaining question is how to impose boundary conditions on~\eqref{eq: dimRedRMA}, or equivalently~\eqref{eq: dimRedHomMA}.  This question was answered in the case $n+1=2$  in \cite{Collins-Li} by the first author and Li by exploiting an inductive structure in the problem.  The relevant boundary data, given in terms of $\phi$ or $v$ is  
    \begin{equation}\label{eq: geomBndData}
    \sum_{i=1}^{n+1}\frac{\del \phi}{\del x_i}\big|_{\del \mathbb{R}^{n+1}_{+}}=0, \text{ or, equivalently } v^{\star}\big|_{\del P} =0.
    \end{equation}
    We remark that~\eqref{eq: geomBndData} is the natural generalization of the boundary data introduced in \cite{Collins-Li} in the case $n+1=2$.  However, it is not at all clear whether this boundary data is appropriate for $n+1>2$.  The current work can be viewed as giving strong evidence that this choice of boundary data is correct.  To illustrate this, let us briefly discuss the inductive nature of the existence problem for Calabi-Yau metrics on $X\setminus D$.  We make the basic observation that if $J = (i_1, \ldots, i_\ell) \subset \{1,\ldots, n+1\}$ is an index set of length $\# J<n+1$, then
    \[
    X_{J} := \bigcap_{j\in J} D_{j}
    \]
    is a Fano variety, and $D^{J}:=\sum_{i \in J^c} (D_i\cap X_{J}) \in |-K_{X_{J}}|$ is an anti-canonical divisor with simple normal crossings.  In terms of the variables $(x_i,\ldots, x_{n+1})$, the region near $X_{J}$, but away from $X_{K}$ for $\# K > \# J$ corresponds to the region 
    \[
    \{\frac{x_{\ell}}{x_{j}} \ll 1 : \ell \in J^{c}, j \in J\} \cap \{\frac{x_{\ell}}{x_{j}} \sim O(1): \ell, j \in J\}
    \]
   Inductively, it is natural to expect that the Calabi-Yau metric on $X_{J}$ (assuming it exists) should determine the asymptotics of $\phi$ near $X_{J}$.  Note that in terms of the function $u$, the region near $X_{J}$ corresponds to the region near the interior of the dimension $\# J^c$ faces of $\del P$.  In the case that $\# J=n$, $X_{J}$ is Fano of dimension $k+1$ and $D^{J}$ consists of a single smooth, anti-canonical divisor for which a Calabi-Yau metric exists by \cite{Tian-Yau}.  In this case, the inductive structure predicts that near $ \{ x_i \ll 1\} \cap \{x_j \gg x_i \,\, \forall j\ne i\}$ we should have an expansion
   \[
   \phi = a(x') + b(x')x_i + c(x') x_i^{1+\frac{1}{k+1}} + \text{ higher order}.
   \]
    where $x' = (x_1,\ldots, \hat{x_i},\ldots, x_{n+1})$.  This is precisely the expansion obtained in Theorem~\ref{thm: boundary-regularity}.  In the case considered in \cite{Collins-Li}, ~\eqref{eq: dimRedHomMA} reduces to an ODE and an expansion of this sort can be obtained directly from power series techniques.  In the general case, it remains to address the regularity near the lower dimensional faces which seems to depend heavily on regularity theory for optimal transport with degenerate densities in polyhedral domains.  In turn, such a regularity theory seems closely related to establishing regularity of the free boundary in~\eqref{eqn: legendre-transform-equation}.  A key result would be to resolve the following conjectural Liouville theorem generalizing the results of \cite{Jhaveri-Savin}. 

    \begin{conj}[Liouville Theorem]\label{conj: Liouville}
        Let $\varphi:(\mathbb R_+)^l\times \mathbb R^{m}\to \mathbb R_{\geq 0}$ be a solution to the Monge-Ampere equation
        \begin{equation}
            \begin{cases}
            \left(\sum_{i = 1}^l\frac{\partial \varphi}{\partial x_i}\right)^k\det D^2\varphi = 1\text{ in }(\mathbb R_+)^l\times \mathbb R^{m}\\
            \sum_{i = 1}^l\frac{\partial \varphi}{\partial x_i} = 0 \text{ on }\partial (\mathbb R_+)^l\times \mathbb R^{m}.
            \end{cases}
        \end{equation}
        Then $\varphi$ is of the form
        \[\varphi(x) = c+\sum_{i = l+1}^{l+m}v_ix_i+\sum_{i,j = 1}^lP_{ij}x_ix_j+p \varphi_{l, k}(x_1, \ldots, x_l)\]
        where $p>0$ and $c, v_i$ are constants, $P_{ij}$ is a positive-definite $l\times l$ matrix and $\varphi_{l, k}$ is the homogenous solution to 
        \begin{equation}
            \begin{cases}
            \left(\sum_{i = 1}^l\frac{\partial \varphi_{l, k}}{\partial x_i}\right)^k\det D^2\varphi_{l, k} = 1\text{ in }(\mathbb R_+)^l\\
            \sum_{i = 1}^l\frac{\partial \varphi_{l, k}}{\partial x_i} = 0 \text{ on }\partial (\mathbb R_+)^l
            \end{cases}
        \end{equation}
        obtained from Corollary~\ref{cor: intro-solution-to-generalized-ansatz} with $n+1 = l$. 
    \end{conj}
	A different, but closely related question is to understand the behavior of degenerating Calabi-Yau hypersurfaces in $\mathbb P^n$. In the situation when the hypersurfaces degenerate into two pieces, the results of \cite{SunZhang, HSVZ} give a complete description of the resulting behavior of the Calabi-Yau metrics. Recently, Y. Li has shown that when the hypersurfaces degenerate to $k$ pieces for $k<n$, the potentials of the Calabi-Yau metrics converge to a solution of a different optimal transport problem on a simplex with degenerate densities \cite{Li-intermediate}. Conjecture~\ref{conj: Liouville} also relates the solutions of \eqref{eq: genearlized-calabi-ansatz} to the boundary behavior of such kind of optimal transport problem. 
	
	\section{Monge-Amp\`ere functionals and their variations}\label{sec: MAFunctionals}
	Let $v:\overline P\to \mathbb R_+$ be a positive convex function and we denote by $u:\mathbb R^n\to \mathbb R$ its Legendre transform, 
	\[u(x) := \sup_{y\in \overline P}(\langle x, y\rangle-v(y)).\]
	Note that since $v$ is positive, we have
	\[u(0) = -\inf v<0, \]
	and therefore the set $\Omega =\{u<0\}$ is a convex set containing the origin. Let us denote the set of all $v$ by 
	\[\mathcal C^{+}(\overline P) = \{v:\overline P\to \mathbb R_+\big| v \text{ is convex }\}.\] 
	The Legendre dual of this class of functions is the class of convex functions 
	\[\mathcal C_{P}(\mathbb R^n) = \{u:\mathbb R^n\to \mathbb R\big| \phi_P-C \leq u < \phi_P \text{ for some } C>0\} \]
	and $\phi_P$ is the support function of $P$. 
	
	For any $v\in \mathcal C^+$, consider the functionals
	\[I(v) = \frac{1}{k+1}\int_{\mathbb R^n}(-u)_+^{k+1}(x)dx\]
	and 
	\[J(v) = \frac{1}{n+1}\int_P\frac{dy}{(v(y))^{n+1}}. \]
	Let $v_t\in \mathcal C^+$ be a smoothly varying family of functions, then the first variations of $I$ and $J$     are given by
	\[\frac{d}{dt}I(v_t) =  -\int_{\mathbb R^n}\dot u(-u)_+^k\]
	and
	\[\frac{d}{dt}J(v_t) = -\int_{P}\frac{\dot v}{v^{n+2}}. \]
	\begin{lem}\label{lem: concavity-of-J}
		The functional $J^{-\frac{1}{n+1}}(v)$ is concave function on $\mathcal C^+(\overline P)$. 
	\end{lem}
	\begin{proof}
		Let $v_t = v_0+tv'$ be a linear path of functions in $\mathcal C^+(\overline P)$. Then we can compute
		\begin{equation}
			\frac{d}{dt}J(v_t) = -\int_{P}\frac{v'}{v_t^{n+2}}
		\end{equation}
		and 
		\begin{equation}
			\frac{d^2}{dt^2}J(v_t) = (n+2)\int_{P}\frac{(v')^2}{v_t^{n+3}}. 
		\end{equation}
		It follows from the Cauchy-Schwartz inequality that
		\[\left(\frac{d}{dt}J(v_t)\right)^2\leq \frac{n+1}{n+2}J(v_t)\left(\frac{d^2}{dt^2}J(v_t)\right), \]
		which implies that $J^{-\frac{1}{n+1}}$ is concave. 
	\end{proof}

	\begin{lem}\label{lem: J-lower-bound}
		For any convex $v:\overline P\to \mathbb R_+$, we have
		\[(n+1)J(v)\geq \frac{|\{u<0\}^{\circ}|}{(-\inf u)}. \]
	\end{lem}
	\begin{proof}
		If we assume $u$ is smooth and strictly convex, then by \cite[Theorem 2.1]{Tong-Yau}, we have
		\[\int_{\{u<0\}}\frac{-u\det D^2u}{(u^{\star})^{n+1}} = |\{u<0\}^{\circ}|. \]
		From this, it follows by a change of variables formula that
		\[\frac{|\{u<0\}^{\circ}|}{(-\inf u)}\leq \int_{\{u<0\}}\frac{\det D^2u}{(u^{\star})^{n+1}} = \int_{\nabla u(\{u<0\})}\frac{1}{v^{n+1}}\leq \int_{P}\frac{1}{v^{n+1}}. \]
		For a general convex $u$ which is not necessarily smooth and strictly convex, we can consider the following functions as a smooth approximation for $u$,  
		\[u_{\eps}(x) := u\star\rho_{\eps}(x)+\eps\phi(x)\]
		where $\rho_{\eps}$ is the standard mollifier with support $B_{\eps}(0)$, and $\phi(x)>0$ is a smooth, strictly convex function with $\nabla \phi(\mathbb R^n) = B_{1}(0)$. Then $u_{\eps}:\mathbb R^n\to \mathbb R$ is smooth and strictly convex function with $\nabla u_{\eps}(\mathbb R^n) = P_{\eps}$ where $P_{\eps} = P+B_{\eps}$ and $+$ is the Minkowski sum. If we let $v_{\eps}(y) = \sup_{x\in\mathbb R^n}(\langle x, y\rangle-u_{\eps}(x))$  be the Legendre transform of $u_{\eps}$, then $v$ is smooth and strictly convex and we have
		\[\int_{P_{\eps}}\frac{1}{v_{\eps}^{n+1}}\geq \frac{|\{u_{\eps}<0\}^{\circ}|}{(-\inf u_{\eps})}. \]
		We would like to take the limit as $\eps\to 0$. As $\eps\to 0$, the right-hand side converges to $\frac{|\{u<0\}^{\circ}|}{(-\inf u)}$. For the left-hand side, we note that since $u_{\eps}(0)\to u(0)$, $v_{\eps}$ is uniformly bounded from below and hence $\frac{1}{v_{\eps}^{n+1}}$ is uniformly bounded from above, we can also easily check that $\frac{1}{v_{\eps}^{n+1}(y)}$ converges to $\frac{1}{v^{n+1}(y)}$ pointwise for all $y$, and hence by the dominated convergence theorem, we obtain that 
		\[\lim_{\eps\to 0}\int_{P_{\eps}}\frac{1}{v_{\eps}^{n+1}} = \int_P\frac{1}{v^{n+1}}\]
		and we are done.  
	\end{proof}

	\begin{prop}\label{prop: min-over-shifts}
		The function $f(x) = J(v-\langle x, \cdot \rangle)$ is strictly convex function defined on a convex set $\Omega = \{u<0\}$ that goes to $\infty$ at $\partial \Omega$, and moreover, $f(x)$ attains its minimum at a unique point $x_0 \in \Omega$ at which 
		\begin{equation}\label{eq: eqn-for-min}
			\int_P\frac{y}{(v(y)-\langle x_0, y\rangle)^{n+2}}dy = 0.
		\end{equation}
	\end{prop}
	\begin{proof}
		For any $x_0\in \mathbb R^n$, let us denote $v_{x_0}(y) := v(y)-\langle x_0, y\rangle$. The Legendre transform of $v_{x_0}$ is given by $u_{x_0}(x) = u(x+x_0)$. For $x_0\in \Omega$, we have $u_{x_0}(0)<0$, therefore $v_{x_0}(y)\geq -u_{x_0}(0)>0$ hence $f(x)$ is well-defined and finite for $x\in\Omega$. We also have $\{u_{x_0}<0\} = \Omega-x_0$, hence by Lemma~\ref{lem: J-lower-bound}, $f(x)$ must go to $\infty$ as $x$ approach $\partial\Omega$. To see that $f(x)$ is convex, let $v_t(y) = v(y)-t\langle x_0, y \rangle$, and let $f(t) := \frac{1}{n+1}\int_P\frac{1}{v_t^{n+1}}$, then
		\begin{equation}
			\frac{d^2 f(t)}{dt^2} = (n+2)(n+1)\int_{P}\frac{\langle x_0, y\rangle^2}{v_t(y)^{n+3}}\,dy>0
		\end{equation}
		hence $f(t)$ is strictly convex. Since $f(x)$ is a strictly convex twice differentiable function defined on a convex set $\Omega$ and it goes to $\infty$ at the boundary, therefore it must have a unique critical point at which it achieves its minimum in $\Omega$. Equation~\eqref{eq: eqn-for-min} follows from taking the derivative of $f$ at a critical point. 
	\end{proof}
	Let us define the space 
	\[\mathcal C_{n+2}(P) := \{v:\overline P\to \mathbb R_+\big| \int_P\frac{y}{v^{n+2}}dy = 0\}\]

	\begin{prop}\label{prop: J-upper-bound}
		For any convex $v:\overline P \to \mathbb R_+$, we have
		\[J(v)\leq \frac{n+2}{n+1}\frac{|\Omega^{\circ}|}{|u(0)|}. \]
	\end{prop}
	\begin{proof}
		By the definition of $\Omega$ and convexity, we have
		\[v(x)\geq \sup \{\phi_{\Omega}(x), |u(0)|\} \]
		where $\phi_{\Omega}(x) = \sup_{y\in\Omega}\langle x, y\rangle$ is the support function of $\Omega$. Then we have
		\begin{align}
			J(u) &= \frac{1}{n+1}\int_{P}v(x)^{-n-1}\,dx\\
			&\leq \frac{1}{n+1}\int_{\mathbb R^n}\sup \{\phi_{\Omega}(x), |u(0)|\}^{-n-1} \,dx\\
			& = \frac{1}{n+1}\int_{\mathbb R^n\setminus |u(0)|\Omega^{\circ}}\phi_{\Omega}(x)^{-n-1}\,dx+\frac{1}{n+1}\int_{|u(0)|\Omega^{\circ}}|u(0)|^{-n-1} \,dx\\
			&\leq \frac{|\Omega^{\circ}|}{n+1}\int_0^{|u(0)|^{-n-1}}t^{\frac{-n}{n+1}}dt+\frac{1}{n+1}\frac{|\Omega^\circ|}{|u(0)|}\\
			& = \frac{n+2}{n+1}\frac{|\Omega^{\circ}|}{|u(0)|}
		\end{align}
	\end{proof}
	
	\begin{corr}\label{cor: balanced-Omega}
		If $v\in \mathcal C_{n+2}(P)$, then
		\[J(v)\leq \frac{C(n)}{|\Omega|(-\inf u)}, \]
		where $C(n)$ is some constant depending only on dimension. Moreover $\Omega = \{u<0\}$ is balanced about the origin. 
	\end{corr}
	
	\begin{proof}
		By Lemma~\ref{lem: J-lower-bound}, we  have
		\[(n+1)J(v)\geq \frac{|\Omega^{\circ}|}{(-\inf u)}. \]
		On the other hand by John's Lemma, there exist ellipsoids $E$ centered at some point $x_0\in\Omega$ such that $E\subset \Omega\subset nE$. By Proposition~\ref{prop: min-over-shifts} and Proposition~\ref{prop: J-upper-bound} applied to the function $\tilde v(y) = v(y)-\langle x_0, y\rangle$ gives
		\[J(v)\leq J(\tilde v)\leq  \frac{n+2}{n+1}\frac{|(\Omega-x_0)^{\circ}|}{|u(x_0)|}\leq \frac{C(n)}        {|\Omega|(-\inf u)}. \]
		If we combine the two inequalities, we have
		\[|\Omega||\Omega^{\circ}|\leq C(n),\]
		which implies that $\Omega$ is balanced about the origin. 
	\end{proof}

	\section{Proof of Theorem~\ref{thm: main-theorem}}\label{sec: proofOfMain}
	\subsection{Solution of a variational problem}
	Let us define the functional
	\begin{equation}
		\mathcal E(v):= -\log I(v)+J^{-\frac{1}{n+1}}(v). 
	\end{equation}
 
        \begin{lem}\label{lem: lower-bound-for-E} 
		There exist constants $c>0$ and $C$, depending on $P$, such that for any $v$ satisfying     
		\[\int_P\frac{y}{v^{n+2}(y)}dy = 0,\] 
		we have
		\begin{equation}\label{ineq: E-lower-bdd}
			\mathcal E(v)\geq -\log ((-\inf u))^{k+1}|\Omega|)+c\left((-\inf u))|\Omega|\right)^{\frac{1}{n+1}}-C. 
		\end{equation}
	\end{lem}
	\begin{proof}
		First, we note that $I$ can be bounded from above straightforwardly,
		\[I(u) = \frac{1}{k+1}\int_{\Omega}(-u)^{k+1}\leq \frac{(-\inf u)^{k+1}|\Omega|}{k+1}.\]
		To bound $J$ from above, we use Corollary~\ref{cor: balanced-Omega} to obtain
		\[J(u)\leq \frac{C(n)}{|\Omega|(-\inf u)}. \]
		Therefore we have proved the Lemma.
	\end{proof}

	\begin{theorem}\label{thm: minimize-functional}
		The functional $\mathcal E$ has a minimizer in $\mathcal C_{n+2}(P)$.
	\end{theorem}
	Before we prove the theorem, we will need the following Lemma, which says that the energy $\mathcal E$ decreases when we move all the Monge-Ampere mass of $u$ outside $\bar\Omega$ onto its boundary. 
	
	\begin{lem}\label{lem: remove-mass-outside-Omega}
		Given any $v\in \mathcal C_{n+2}(P)$, there exist $\tilde v\in C_{n+2}(P)$ such that 
		\begin{enumerate}
			\item $I(v) = I(\tilde v)$. 
			\item $J(\tilde v)\geq J(v)$. 
			\item $\tilde v\geq \phi_{\tilde\Omega}$ and $\tilde v = \phi_{\tilde\Omega}$ on $\partial P$. 
		\end{enumerate}
	\end{lem}
	\begin{proof}
		Let $\tilde v(y):= \sup_{x\in \overline\Omega}(\langle x, y\rangle-u(x))+\langle y, x_0 \rangle$, where the $x_0$ is chosen by Proposition~\ref{prop: min-over-shifts} so that 
		\[\int_P\frac{y}{\tilde v(y)^{n+2}}\, dy= 0. \]
		Then $\tilde u(x):= \sup_{y\in P}(\langle x, y\rangle-\tilde v(y))$, and we have the obvious inequalities 
		\[\tilde v(y)\leq v(y)+\langle x_0, y\rangle\]
		\[\tilde u(x)\geq u(x-x_0)\]
		but we also note that $\tilde v$ is the Legendre transform of the convex function 
		\[\hat u(x) = \begin{cases}u(x-x_0)& \text{ if }x\in x_0+\Omega\\
			\infty &\text{otherwise}
		\end{cases}\]
		therefore from the fact that $\hat u^{\star\star} = \hat u$, we obtain $x-x_0\in\Omega$, $u(x-x_0) = \tilde u(x)$. From this, we must have
		\[I(\tilde v) = I(v). \]
		Now we prove claim 2. We note that by Proposition~\ref{prop: min-over-shifts}, we have
		\[J(v)\leq J(v-\langle x_0, \cdot\rangle). \]
		and 
		\[\tilde v(y) = \sup_{x\in \Omega}(\langle x, y\rangle-u(x))+\langle y, x_0 \rangle\leq \sup_{x\in \mathbb R^n}(\langle x, y\rangle-u(x))+\langle y, x_0 \rangle = v(y)+\langle y, x_0 \rangle\]
		therefore we have
		\[J(v)\leq J(v-\langle x_0, \cdot\rangle)\leq J(\tilde v).  \]
		For the third claim, it's clear that $\tilde v\geq \phi_{\tilde\Omega}$, so we must show that $\tilde v(y) = \phi_{\tilde \Omega}(y)$ for $y\in \partial P$. Fix any $y\in \partial P$, and notice that $\tilde v = \hat u^{\star}$ on $P$, therefore it suffices to show that the supporting hyperplane $\langle x ,y\rangle-\tilde v(y)$ of $\hat u$ touches $\hat u$ at a boundary point $x\in \partial(x_0+\Omega)$, then we must have
		$\tilde v(y) = \langle x, y\rangle-u(x) = \langle x, y\rangle \leq  \phi_{\tilde\Omega}(y)$ which implies $\tilde v(y) = \phi_{\tilde \Omega}(y)$. We prove this by contradiction, suppose $\langle x ,y\rangle-\tilde v(y)$ touches $\hat u$ only at interior points, then $y$ must be in the interior of $\nabla \hat u(\Omega)$, but $y\in \partial P$ and $\nabla \hat u(\Omega)\subset \overline P$, which is a contradiction, and therefore we've proven the Lemma. 
	\end{proof}

	\begin{proof}[Proof of Theorem~\ref{thm: minimize-functional}]
		Let $v_j\in \mathcal C_{n+2}(P)$ be a minimizing sequence of functions such that
		\[\mathcal E(v_j)\to \inf_{v\in C_{n+2}(P)}\mathcal E(v), \]
		and let $u_j: \mathbb R^n\to \mathbb R$ be the Legendre transform of $v_j$, and denote $\Omega_j:= \{x: u_j(x)<0\}$.
		By Lemma~\ref{lem: remove-mass-outside-Omega}, we can assume without loss of generality that $v_j$ satisfy $v_j\geq \phi_{\Omega_j}$. By Lemma~\ref{lem: lower-bound-for-E}, we know that 
		\begin{align}C &\geq\mathcal E(v_j)\\
			&\geq -\log ((-\inf u_j)^{k+1}|\Omega_j|)+c((-\inf u_j)|\Omega_j|)^{\frac{1}{n+1}}-C\\
			&\geq -\log ((-\inf u_j)^{k+1}|\Omega_j|)+c'((-\inf u_j)^{k+1}|\Omega_j|)^{\frac{1}{n+k+1}}-C
		\end{align}
		where the last line follows from the ABP maximum principle bound $|\Omega_j|\geq c(n)|P|^{-1}(-\inf u_j)^n$. From this we obtain a uniform bound
		\begin{equation}\label{eq: I(u_j)-bound}
			C^{-1}\leq (-\inf u_j)^{k+1}|\Omega_j|\leq C,
		\end{equation}
		which implies that $|-\log I(u_j)|\leq C$ is uniformly bounded. From this we also obtain the bound $J^{-\frac{1}{n+1}}(u_j)\leq C$, which implies 
		\begin{equation}\label{eq: J-lower-bound}
			(-\inf u_j)|\Omega_j|\leq C. 
		\end{equation}
		From \eqref{eq: I(u_j)-bound}, \eqref{eq: J-lower-bound} and the ABP bound, we obtain
		\begin{equation}
			C^{-1} \leq -\inf u_j \leq C
		\end{equation}
		and 
		\begin{equation}
			C^{-1} \leq |\Omega_j| \leq C. 
		\end{equation}
		and moreover from Corollary~\ref{cor: balanced-Omega}, we also know that $\Omega_j$ is balanced about the origin, thus we have 
		\[0<c\leq c(-\inf u_j)\leq |u_j(0)|\leq C(-\inf u_j)\leq C\]
		thus there exist a subsequence for which $\lim_{j\to \infty}u_j(0)\to -c<0$ which implies a lower bound for $v_j$. Moreover since $|\nabla u_j|\leq C$, there exist a small $r$ such that $B_r(0)\subset\Omega_j$ and therefore by the volume upper bound we also have 
		\[B_r(0)\subset\Omega_j\subset B_R(0). \]
		From this, we can see that $v_j$ is bounded from above as well. Therefore, we can extract a convergent subsequence that uniformly converges (this follows from the gradient bounds) $v_j\to v_{\infty}$, and it is clear that $u_j\to u_{\infty} = v_{\infty}^{\star}$ as well. $v_{\infty}$ must be a minimizer of $\mathcal E$. 
    \end{proof}
    \subsection{Solution of \eqref{eqn: main-equation}}
	Now we show that the solution of the variational problem that we obtained is a solution of \eqref{eqn: main-equation}. Let $v\in \mathcal C_{n+2}(P)$ be the minimizer of $\mathcal E$ in $\mathcal C_{n+2}(P)$ from Theorem~\ref{thm: minimize-functional}, and $u(x) = \sup_{y\in P}(\langle x, y\rangle-v(y))$ be its Legendre tranform. Define the probability measures $d\mu$ supported on $P$ and $d\nu$ supported on $\Omega$ by
	\begin{equation}
		d\mu := \frac{(-u)_+^{k}dx}{\int_{\mathbb R^n}(-u)_+^{k} }
	\end{equation}
	and 
	\begin{equation}
		d\nu = \frac{\frac{dy}{v^{n+2}(y)}}{\int_Pv^{-(n+2)}}. 
	\end{equation}
	
	\begin{prop}\label{prop: naba u-solves-optimization-problem}
				For any bounded convex function $\hat v:\overline P\to \mathbb R$, let $\hat u(x)= \sup_{y\in \overline P}(\langle x, y\rangle-\hat v(y))$ be its Legendre transform, then we have
				\begin{equation}
					\int_{P}v d\nu+\int_{\Omega}ud\mu\leq \int_{P}\hat v d\nu+\int_{\Omega}\hat ud\mu. 
				\end{equation}
	\end{prop}
	\begin{proof}
		Notice that since $d\mu$ and $d\nu$ are both probability measures, hence the right-hand side is unchanged by adding a constant to $\hat v$. Without loss of generality, we can then assume $\hat v>0$ by adding a large enough positive constant.
		\begin{lem}
			There is a linear function $l(y) = t_0+\langle x_0, y\rangle$ such that $\tilde v(y)=\hat v(y)+l(y)$ satisfies
			\begin{enumerate}
				\item $\tilde v >0$. 
    			\item $\tilde v\in \mathcal C_{n+2}$. 
				\item $J(\tilde v) = J(v)$. 
			\end{enumerate}
		\end{lem}
		\begin{proof}[Proof of Lemma]
			For $t\in (\inf \hat u, \infty)$, let $\tilde v_t(y) = \hat v(y)+t-\langle y, x_t\rangle$, where $x_t$ is chosen according to Proposition~\ref{prop: min-over-shifts}, and from this it's easy to see that $x_t$ is continuous with respect to the parameter $t$. Then $J(v_t)$ is a continuous function of $t$ and  by Proposition~\ref{prop: min-over-shifts}, we have
			\[\lim_{t\to \infty}J(\tilde v_t)\leq \lim_{t\to \infty}J(\hat v+t) = 0. \]
			If we let $\Omega_t = \{u<t\}$, then by Lemma~\ref{lem: J-lower-bound}, we have
			\[\lim_{t\to \inf \hat u}J(\tilde v_t)\geq c(n)\lim_{t\to \inf \hat u}\frac{|(\Omega_t-x_t)^{\circ}|}{(-\inf \hat u+t)} =\infty.\]
			Therefore by the intermediate value theorem, there exist $t_0\in (\inf \hat u, \infty)$ such that $J(\tilde v_{t_0}) = J(v)$. This proves the Lemma. 
		\end{proof}
		Let $\tilde v = \hat v+t_0-\langle x_0, \cdot \rangle$ be chosen as in the preceding Lemma. Since $v$ minimizes the energy $\mathcal E$ over the space $\mathcal C_{n+2}(P)$, we have
		\[-\log I(v)+J^{-\frac{1}{n+1}}(v)\leq -\log I(\tilde v)+J^{-\frac{1}{n+1}}(\tilde v)\]
		and since $J(v) = J(\tilde v)$, we get $I(\tilde v)\leq I(v)$. 
        
        By the convexity of the functions $\mathbb R^+\ni t\mapsto t^{-(n+1)}$ and $\mathbb R\ni s\mapsto (-s)_+^{k+1}$, we have
		\[(n+1)(v-\tilde v)v^{-(n+2)}\leq \tilde v^{-(n+1)}-v^{-(n+1)}\]
		and
		\[(k+1)(u-(\hat u-t_0))(-u)_+^k\leq (-\hat u+t_0)_+^{k+1}-(-u)_+^{k+1}\]
		and integrating this gives the inequalities
		\begin{equation}\label{eq: convex-J}
			\int_P(v-\tilde v)d\nu\leq \frac{J(\tilde v)-J(v)}{\int_Pv^{-(n+2)}} = 0
	\end{equation}
		and \begin{equation}\label{eq: convex-I}
			\int_{\Omega}(u-(\hat u-t_0))d\mu\leq \frac{I(\hat u-t_0)-I(u)}{\int_\Omega (-u)_+^{k}} = \frac{I(\tilde u)-I(u)}{\int_\Omega (-u)_+^{k}}\leq 0. 
	\end{equation}
        Adding these two inequalities gives 
        \begin{equation}
            \int_P(v-\tilde v)d\nu+\int_{\Omega}(u-\hat u)d\mu+t_0\leq 0
        \end{equation}
    Now notice that by Proposition~\ref{prop: min-over-shifts}, $d\nu$ has barycenter at the origin, adding a linear function to $\tilde v$ does not change the left-hand side of equation~\ref{eq: convex-J}, therefore we have
    \begin{equation}
        \int_P(v-\hat v)d\nu-t_0=\int_P(v-(\hat v+t_0))d\nu = \int_P(v-\tilde v)d\nu. 
        \end{equation}
        The proposition follows by taking the sum of the two preceding inequalities.
\end{proof}
	
	\begin{theorem}\label{thm: u-solve-equation}
		$v\in C^{1, \alpha}(\overline P)\cap C^{\infty}(P)$ and satisfies the equation~\eqref{eqn: main-equation}. 
	\end{theorem}
	\begin{proof}
        Proposition~\ref{prop: naba u-solves-optimization-problem} implies that $\nabla u$ is the solution to an optimal transport problem from the measure $d\mu$ on $P$ to the measure $d\nu$ on $\mathbb R^n$ \cite{Gangbo-McCann}. By \cite{Caffarelli}, we know that $u$ solves \eqref{eqn: legendre-transform-equation} in the Alexandrov sense, and by the regularity theory of Caffarelli \cite{Caffarelli}, we know that $u\in C^{\infty}(\Omega)$ and $v\in C^{\infty}(P)$. Moreover by \cite[Theorem 1.1]{Jhaveri-Savin}, we see that $u\in C^{1, \alpha}(\mathbb R^n)$ and $v\in C^{1, \alpha}(\overline P)$ for some $\alpha>0$. 
	\end{proof}
	
	\begin{proof}[Proof of Theorem~\ref{thm: main-theorem}]
		This is simply the combination of Theorem~\ref{thm: minimize-functional} and Theorem~\ref{thm: u-solve-equation}. 
	\end{proof}

	\section{Regularity}\label{sec: regularity}
     In this section, we study the boundary regularity of $v$ and prove Theorem~\ref{thm: boundary-regularity}. An important observation here is that near a smooth point of $\partial P$, a homogenization of $v$ solves an optimal transport problem with degenerate densities between domains with {\em flat} boundaries. Hence the results of \cite{Jhaveri-Savin} can be readily applied. 

	Let $P\subset \mathbb R^n$ be a open convex set, we denote $C(P)\subset \mathbb R^{n+1} $ to be the affine cone over $P$ defined by
 \[C(P) = \{(\lambda y, \lambda)\in \mathbb R^{n+1}: \lambda>0, y\in P\}.\] 
 Then given any function $v:P\to \mathbb R_+$ and $m \in \mathbb R$, we can let $\varphi:C(P)\to \mathbb R_+$ be its homogenization of degree $m$ given by
    \begin{equation}\label{eq: homogenization}
        \varphi(\lambda y, \lambda) := (\lambda v(y))^m. 
    \end{equation}
     The following proposition is a calculation that shows if $v$ satisfy equation~\eqref{eqn: main-equation}, then $\varphi$ satisfy an equation closely related to \eqref{eq: genearlized-calabi-ansatz}. 
	\begin{prop}\label{prop: phi-solves-gen-Calabi-ansatz}
		Given a solution $v$ of \eqref{eqn: main-equation}, Let $\varphi$ its homogenization of degree $m = 1+\frac{n+1}{n+k+1}$, then $\varphi$ satisfy the equation
		\begin{equation}
				\left(\frac{\partial \varphi}{\partial x_{n+1}}\right)^{k}\det D^2\varphi  = \left(\frac{n+1}{n+k+1}\right)\left(\frac{2n+k+2}{n+k+1}\right)^{n+k+1}
		\end{equation}
		with boundary value 
		\[\frac{\partial \varphi}{\partial y_{n+1}} = 0 \text{ on }\partial C(P)\]
		Moreover \[\nabla \varphi(C(P)) = \mathbb R^{n+1}\cap\{y_{n+1}>0\}.\] 
	\end{prop}
	\begin{proof}
		For $i, j \in \{1, \ldots, n\}$, the derivative of $\varphi$ is given by
		\[\varphi_i =m\left(y_{n+1}v(\frac{y_1}{y_{n+1}}, \ldots, \frac{y_n}{y_{n+1}})\right)^{m-1}v_i \]
		\[\varphi_{n+1} = m\left(y_{n+1}v(\frac{y_1}{y_{n+1}}, \ldots, \frac{y_n}{y_{n+1}})\right)^{m-1}\left(v-v_i\frac{y_i}{y_{n+1}}\right) = m\frac{\varphi}{y_{n+1}}-\varphi_i\frac{y_i}{y_{n+1}} \]
		the second derivatives of $\varphi$ is given by
		\begin{align}
			\varphi_{ij} &= m\left(y_{n+1}v(\frac{y_1}{y_{n+1}}, \ldots, \frac{y_n}{y_{n+1}})\right)^{m-1}\frac{v_{ij}}{y_{n+1}}\\ 
			&\qquad +(m-1)m\left(y_{n+1}v(\frac{y_1}{y_{n+1}}, \ldots, \frac{y_n}{y_{n+1}})\right)^{m-2}v_iv_j\\
			& = m\left(y_{n+1}v(\frac{y_1}{y_{n+1}}, \ldots, \frac{y_n}{y_{n+1}})\right)^{m-2}\left(vv_{ij}+(m-1)v_i v_j\right)
		\end{align}
		\begin{align}
			\varphi_{i(n+1)} &= m\left(y_{n+1}v(\frac{y_1}{y_{n+1}}, \ldots, \frac{y_n}{y_{n+1}})\right)^{m-1}\left(-v_{ik}\frac{y_k}{y_{n+1}^2}\right)\\
			&\qquad +(m-1)m\left(y_{n+1}v(\frac{y_1}{y_{n+1}}, \ldots, \frac{y_n}{y_{n+1}})\right)^{m-2}v_i\left(v-v_k\frac{y_k}{y_{n+1}}\right)\\
		\end{align}
		\begin{align}
			\varphi_{(n+1)(n+1)} &= m\left(y_{n+1}v(\frac{y_1}{y_{n+1}}, \ldots, \frac{y_n}{y_{n+1}})\right)^{m-1}\left(v_{ik}\frac{y_iy_k}{y_{n+1}^3}\right)\\
			&\qquad +(m-1)m\left(y_{n+1}v(\frac{y_1}{y_{n+1}}, \ldots, \frac{y_n}{y_{n+1}})\right)^{m-2}\left(v-v_k\frac{y_k}{y_{n+1}}\right)^2\\
		\end{align}
		Therefore, we have
		\begin{align}
			\det D^2\varphi  &= (\det \varphi_{ij})(\varphi_{(n+1)(n+1)}-\varphi^{ij}\varphi_{(n+1)i}\varphi_{(n+1)j})\\
			&= m^{n+1}(y_{n+1}v)^{(m-2)(n+1)}v^n(\det v_{ij})(1+(m-1)v^{-1}v^{ij}v_iv_j)\left(\frac{(m-1) v^2}{1+(m-1)v^{-1}v^{ij}v_iv_j}\right)\\
			& = (m-1)m^{n+1}\varphi^{\frac{(m-2)(n+1)}{m}}v^{n+2}\det v_{ij}
		\end{align}
		and if $\det D^2v = \frac{1}{v^{n+2}(-v^{\star})^k}$ and $m-1 = \frac{n+1}{n+k+1}$, then we obtain
		\begin{align}
			\det D^2\varphi &= (m-1)m^{n+1}\varphi^{\frac{(m-2)(n+1)}{m}}(-v^{\star})^{-k}\\
			& = (m-1)m^{n+k+1}(\varphi_{n+1})^{-k}
		\end{align}
		
	\end{proof}
	
    This proposition allows us to prove Corollary~\ref{cor: intro-solution-to-generalized-ansatz}.

    \begin{proof}[Proof of Corollary~\ref{cor: intro-solution-to-generalized-ansatz}]
    Let $P$ be a standard symmetric simplex with barycenter the origin, by Theorem~\ref{thm: main-theorem}, there is a function $v\in C^{1, \alpha}(\overline P)\cap C^{\infty}(P)$ which solves equation~\eqref{eqn: main-equation}. Moreover by Proposition~\ref{prop: phi-solves-gen-Calabi-ansatz}, if we define $\varphi$ its homogenization as in \eqref{eq: homogenization}, then $\varphi$ solves 
    \begin{equation}
    \begin{cases}
        \left(\frac{\partial \varphi}{\partial x_{n+1}}\right)^k\det D^2\varphi = const \text{ in }C(P)\\
        \frac{\partial \varphi}{\partial x_{n+1}}= 0 \text{ on }\partial C(P). 
    \end{cases}
    \end{equation}
    Since $P$ is the standard symmetric simplex, there is an orthogonal transformation $T:\mathbb R^{n+1}\to \mathbb R^{n+1}$ which maps $C(P)$ to the positive orthant and $T(0, \ldots, 0, 1) = (\frac{1}{\sqrt{n+1}}, \ldots, \frac{1}{\sqrt{n+1}})$. If we let $\tilde\varphi(x) = \varphi(T^{-1}(x))$, then by straightforward calculation, we see that $\tilde\varphi$ solves
    \begin{equation}
    \begin{cases}
        (\sum_{k = 1}^{n+1}\frac{\partial \tilde \varphi}{\partial x_i})^k\det D^2\tilde\varphi = const \text{ in }C(P)\\
        \sum_{k = 1}^{n+1}\frac{\partial \tilde \varphi}{\partial x_i}= 0 \text{ on }\partial C(P). 
    \end{cases}
    \end{equation}
    In particular, $\tilde\varphi$ is a solution of the generalized Calabi ansatz from \cite[Lemma 2.1]{Collins-Li}. By simply rescaling, we obtain a solution of \eqref{eq: genearlized-calabi-ansatz}. 
    \end{proof}
    Let $\varphi$ be the solution to \eqref{eq: genearlized-calabi-ansatz} obtained from the Corollary~\ref{cor: intro-solution-to-generalized-ansatz}. Then the following proposition gives a precise description of the boundary behavior of $\varphi$ near a smooth point of the boundary. 
    \begin{prop}\label{thm: boundary-expansion}
        Let $\bold x\in \partial \mathbb R_+^{n+1}$ be a point on the top dimensional stratum of the boundary. Without loss of generality, let us assume that $\bold x = (0, \bold x_2, \ldots, \bold x_{n+1})$ where $\bold x_{i}>0$ for $i = 2, \ldots, n+1$. Then for any $0<\alpha<1$, there exists a positive definite matrix $P\in Mat_+(n\times n)$ and $p>0$ such that
        \begin{align}\label{eqn: expansion-phi}
        \varphi(x) &= \varphi(\bold x)+\langle\nabla \varphi(\bold x), x-\bold x\rangle + \sum_{i, j=2}^{n+1}P_{ij}(x_i-\bold x_i-x_1+\bold x_1)(x_j-\bold x_j-x_1+\bold x_1) + p(x_1-\bold x_1)^{1+\frac{1}{k+1}}\\
        &\qquad + O\left((|x'-\bold x'|^2+|x_1-\bold x_1|^{1+\frac{1}{k+1}})^{1+\frac{\beta}{2}}\right). 
        \end{align}
        In particular, this implies the restriction of $\varphi$ to the boundary $\partial \mathbb R^{n+1}_+$ is in $C^{2, \alpha}(\partial \mathbb R^{n+1}_+\cap \{x_2>0, \ldots, x_{n+1}>0\})$ for any $0<\beta<1$. 
    \end{prop}
    \begin{proof}
        Consider the change of variables 
        \[(\tilde x_1, \ldots, \tilde x_{n+1}) = (x_1, x_2-x_1, \ldots, x_{n+1}-x_1).\]
        Then with respect to this new coordinate system $(\tilde x_1, \ldots, \tilde x_{n+1})$, $\varphi$ satisfies
        \begin{equation}
		\begin{cases}
			(\det D^2\varphi)(\varphi_{\tilde x_1})^k&=1 \text{ in }\{\tilde x_1>0, \tilde x_2>-\tilde x_1, \ldots, \tilde x_{n+1}>-\tilde x_1\}\\
			\varphi_{\tilde x_1}&= 0 \text{ on }\partial \{\tilde x_1>0, \tilde x_2>-\tilde x_1, \ldots, \tilde x_{n+1}>-\tilde x_1\}. 
		\end{cases}
	\end{equation}
        Since $u$ is strictly convex, there exists a compactly supported section $S_{h}(\bold x)$ of $\bold x$, such that $\partial \{\tilde x_1>0, \tilde x_2>-\tilde x_1, \ldots, \tilde x_{n+1}>-\tilde x_1\}\cap \overline S_h = \overline S_h\cap \{\tilde x_1 = 0\}$. 
        
        Without loss of generality, let us assume that $\tilde{\bold x} = (0, \ldots, 0)$ by a translation of the coordinates. Then by the results of Jhaveri and Savin \cite[Section 4]{Jhaveri-Savin}, we have
        \[\frac 1 M\text{Id}\leq D^2_{\tilde x'}\varphi \leq M\text{Id on }C_r\cap \{\tilde x_1\geq 0\},\]
        and there exist $\sigma>0$ such that 
        \[\frac{\varphi_{\tilde x_1}(\tilde x)-\varphi_{\tilde x_1}(0)}{\tilde x_1^{\frac{1}{k+1}}}\in C^{\sigma}(C_r\cap \{\tilde x_1\geq 0\})\] 
        where $C_r$ is a neighborhood of $0$ of the form
        \[C_r =  (-r^{\frac{k+1}{k+2}}, r^{\frac{k+1}{k+2}})\times B_{r^{1/2}}(0)\]
        for some small enough $r>0$. Denote 
        \[\tilde\varphi(\tilde x) := \varphi(\tilde x)-\varphi(0)-\langle\tilde x, \nabla\varphi(0)\rangle, \]
        which satisfies $\tilde\varphi(0) = |\nabla \tilde\varphi(0)| = 0$.
        Since $f(x) = \frac{\tilde\varphi_{\tilde x_1}(\tilde x)}{\tilde x_1^{\frac{1}{k+1}}}$ is a well-defined H\"older continuous function on $C_r\cap \{\tilde x_1\geq 0\}$, it has a value at $0$ and
        \[|f(x_1, x') - f(0)|\leq C(|x_1|^{\sigma}+|x'|^{\sigma})\]
        which implies 
        \[\tilde\varphi_{\tilde x_1}(\tilde x) = f(0)\tilde x_1^{\frac{1}{k+1}}+O(|\tilde x_1|^{\frac{1}{k+1}+\sigma})+O(|x'|^{\sigma}|\tilde x_1|^{\frac{1}{k+1}}). \]
        Integrating this from $(\tilde x_1, \tilde x')$ to $(0, \tilde x')$ gives
        \[\tilde\varphi(\tilde x_1, \tilde x')-\tilde\varphi(0, \tilde x') = \frac{k+1}{k+2}f(0)\tilde{x_1}^{1+\frac{1}{k+1}}+O(|\tilde x_1|^{1+\frac{1}{k+1}+\sigma})+O(|x'|^{\sigma}|\tilde x_1|^{1+\frac{1}{k+1}}) \]
        and moreover since $\frac 1 M\text{Id}\leq D^2_{\tilde x'}\tilde\varphi \leq M\text{Id}$, we obtain
        \[\frac 1 C|\tilde x'|^2\leq \tilde\varphi(0, \tilde x') \leq C|\tilde x'|^2, \]
        which implies on $C_r\cap \{\tilde x_1\geq 0\}$, we have
        \[\frac 1 C|\tilde x'|^2+\frac{k+1}{k+2}f(0)\tilde{x_1}^{1+\frac{1}{k+1}}+O(|\tilde x|^{\sigma}|\tilde x_1|^{1+\frac{1}{k+1}})\leq \tilde\varphi(\tilde x_1, \tilde x') \leq  C|\tilde x'|^2+\frac{k+1}{k+2}f(0)\tilde x_1^{1+\frac{1}{k+1}}+O(|\tilde x|^{\sigma}|\tilde x_1|^{1+\frac{1}{k+1}}). \]
        If we define the rescalings
        \[\tilde\varphi_t(\tilde x_1, \tilde x') = \frac{\tilde\varphi(t^{\frac{k+1}{k+2}}\tilde x_1, t^{\frac 1 2}\tilde x')}{t}\]
        then $u_t$ also satisfy
        \[(\frac{\partial \tilde\varphi_t}{\partial \tilde x_1})^k\det D^2\tilde\varphi_t = 1, \]
        with 
        \[\frac{\partial \tilde\varphi_t}{\partial \tilde x_1} = 0 \text{ on }C_{rt^{-1}}\cap \{\tilde x_1 = 0\}\]
        and for any $R>0$, we have
        \[\frac 1 C|\tilde x'|^2+\frac{k+1}{k+2}f(0)\tilde x_1^{\frac{k+2}{k+1}}+O(t^{\frac{\sigma}{2}})\leq \tilde\varphi_t \leq C|\tilde x'|^2+\frac{k+1}{k+2}f(0)\tilde x_1^{\frac{k+2}{k+1}}+O(t^{\frac{\sigma}{2}}) \]
        on $C_{R}\cap \{\tilde x_1\geq 0\}$. Hence by \cite{Caffarelli}, there exist a subsequence $t_j\to 0$ such that $\tilde\varphi_{t_j}$ converges in $C^{\infty}_{loc}(C_{R}\cap \{\tilde x_1>0\})$ for any $R>0$, to a function $\tilde\varphi_{0}:\mathbb R_+ \times \mathbb R^n$ which satisfies 
        \[(\frac{\partial \tilde\varphi_0}{\partial \tilde x_1})^k\det D^2\tilde\varphi_0 = 1, \]
        Moreover since
        \[\frac{\frac{\partial \tilde\varphi_t}{\partial \tilde x_1}}{\tilde x_1^{\frac{1}{k+1}}} = \frac{\tilde\varphi_{\tilde x_1}(t^{\frac{k+1}{k+2}}\tilde x_1, t^{\frac 1 2}\tilde x')}{(t^{\frac{k+1}{k+2}}\tilde x_1)^{\frac{1}{k+1}}} = f(t^{\frac{k+1}{k+2}}\tilde x_1, t^{\frac 1 2}\tilde x'), \]
        it follows by taking a limit as $t\to 0$ that 
        \[\frac{\partial \tilde\varphi_0}{\partial \tilde x_1} = f(0)\tilde x_1^{\frac{1}{k+1}}, \]
        which implies by the Liouville theorem \cite[Theorem 1.4]{Jhaveri-Savin} that 
        \[\tilde\varphi_0(x) = P(\tilde x')+\frac{k+1}{k+2}f(0)\tilde x_1^{\frac{k+2}{k+1}}\]
        where $P$ is a uniformly convex quadratic function. It follows that for any $\eps>0$, there exist sufficient small $t_j$, such that
        \[\sup_{C_1\cap \{\tilde x_1\geq 0\}}|\tilde\varphi_0-\tilde\varphi_{t_j}|\leq \eps, \]
        and hence by \cite[Proposition 5.1]{Jhaveri-Savin}, for any $0<\beta<1$, there exist a linear transformation $A = \text{diag}(a_1, A')$ such that
        \[\tilde\varphi_t(A\tilde x) = \tilde\varphi_0(\tilde x)+O\left((|\tilde x'|^2+|\tilde x_1|^{1+\frac{1}{k+1}})^{1+\frac{\beta}{2}}\right). \]
        for $t>0$ sufficiently small, which proves the result. 
    \end{proof}
	Given the expansion of $\varphi$, we are ready to prove Theorem~\ref{thm: boundary-regularity}. 
	
	\begin{proof}[Proof of Theorem~\ref{thm: boundary-regularity}]
            We define $\varphi$ to be the homogenous extension of $v$ as in Proposition~\ref{prop: phi-solves-gen-Calabi-ansatz}, then we see that 
            \[v(y_1, \ldots, y_n) = \varphi^{\frac{n+1+k}{2n+2+k}}(y_1, \ldots, y_n, 1-\sum_{i=1}^ny_i) \]
            where $\varphi$ satisfy equation~\ref{eq: genearlized-calabi-ansatz}. By Proposition~\ref{thm: boundary-expansion}, we know that $\varphi$ admits an expansion of the form \eqref{eqn: expansion-phi} near $\bold x$. Plugging this expansion into the equation above gives 
            \begin{equation}
                v(y) = v(\bold y)+\nabla v(\bold y)\cdot(y-\bold y)+p'x_1^{1+\frac{1}{k+1}}+ \sum_{i, j = 2}^nQ_{ij}(y_i-\bold y_i)(y_j-\bold y_j)+O\left(\left(|y|^2+y_1^{1+\frac{1}{k+1}}\right)^{1+\frac{\alpha}{2}}\right)
            \end{equation}
            for some $\alpha>0$. 
	\end{proof}	
        Finally, we prove a basic global convexity property for the free boundary.
        \begin{prop}
            Suppose $u$ solves~\eqref{eqn: legendre-transform-equation}.  Then the free boundary $\{u=0\}$ is strictly convex, in the sense that it does not contain a line segment.
            \end{prop}
            \begin{proof}
                This result is closely related to a result of Savin~\cite{SavinObs} concerning the structure of the free boundary in the obstacle problem.  We sketch the details for the reader's convenience.  First, it is straightforward to check directly that the measure
                \[
                d\mu = (u^\star)^{n+2}(-u)^k_{+}
                \]
                is doubling on $\{u\leq 0\}$.  Precisely, there is a constant $C>1$, depending on $n,k$ and a constant $A>0$
                \[
                A^{-1} \leq \inf_{\{u <0\}} u^{\star} \leq \sup_{\{u<0\}} u^{\star} \leq A
                \]
                such that, for any convex set $K$ with center of mass $x_0 \in \{u\leq 0\}$ we have
                \[
                \mu(K) \leq C\mu(\frac{1}{2}K)
                \]
                where $\frac{1}{2}K$ denotes the rescaling of $K$ about $x_0$; see \cite[Lemma 2.6]{Jhaveri-Savin} for a general result.  Since $u$ is bounded in compact sets and the graph of $u$ does not contain a line, it follows from \cite[Lemma 1]{Caffarelli2} that, for any point $x_0\in \{x=0\}$ and any $h>0$ we can find $p_h \in \mathbb{R}^n$ such that the section
                \[
                S_{h}^c(x_0) = \{ y: u(y) < h+ p_h\cdot (y-x_0)\}
                \]
                is non-empty, bounded and has $x_0$ as center of mass.  Thanks to the mass doubling property, we obtain from \cite[Lemma 3]{Caffarelli2} (see also \cite[Lemma 2.7]{SavinObs}) the following balancing property 
                \begin{lem}\label{lem: balancing}
                    Suppose $x \in S_{h}^{c}(x_0)$ has $u(x)\leq p_{h}\cdot(x-x_0)$, then $S_{h}^{c}(x_0)$ is balanced around $x$ in the sense that
                    \[
                    (1+\eta)x-\eta S_{h}^{c}(x_0) \subset S_{h}^{c}(x_0)
                    \]
                    for $\eta$ depending on the doubling constant and $n$.  In other words, if $[y_1,y_2]$ is a line segment passing through $x$, with $y_i \in \del S_{h}^{c}(x_0)$, then
                    \[
                    \frac{|y_1-x|}{|y_2-x|} >\eta.
                    \]
                \end{lem}
                Suppose, for the sake of a contradiction, that $\{u=0\}$ contains a line segment.  By rotation and translation, we may assume that this segment is given by $te_n$ for $t\in [0,A]$, and $u(te_n)>0$ for $t<0$, and $t>A$.  Let $0<a\ll A$, and let 
                \[
                S_{h}^c(a) = \{y: u(y) < h+ p_{h,a}\cdot (y-ae_n)\}
                \]
                be the section centered at $ae_n$.  First we claim that for $h$ small we have $p_{h,a}\cdot e_n <0$.  Indeed, if $p_{h,a}\cdot e_n \geq 0$, then $Ae_n\in S_{h}^c(a)$, which, by balancing implies $-\eta \frac{A}{2}e_n \in S^c_{h}(a)$.  But $u(-\eta \frac{A}{2}e_n) >\delta$ for some $\delta>0$, contradicting $p_{h,a}\cdot e_n <0$ if $h \ll \delta$.

                Thus, for all $h$ sufficiently small, $S_{h}^{c}(a)$ is balanced around $0$ by Lemma~\ref{lem: balancing}. Define
                \[
                \begin{aligned}
                t_{h,a}&:= \sup\{ t>a : te_n \in S^c_{h}(a)\}\\
                b_{h,a} &:= \sup \{ t>0: -te_{n} \in S^{c}_{h}(a)\}
                \end{aligned}
                \]
                Since $S_{h}^c(a)$ is balanced around $0$ and $t_{h,a}>a$ we see that $b_{h,a} >\eta a >0$ for all $h$ small.  This implies a uniform lower bound $u(b_{h,a})> \delta'>0$.  It follows that $t_{h,a} \rightarrow a$ as $h\rightarrow 0$, contradicting balancing around $a$.

            \end{proof}
            {\bf Acknowledgements:} F.T. is supported by the Center for Mathematical Sciences and Applications at Harvard University, which he would like to thank for its generous support. F.T. is also grateful to Nick McCleerey for helpful discussions. T.C.C is supported in part by NSF CAREER grant DMS-1944952.

\end{document}